\pdfoutput=1
\documentclass[11pt]{article}

\usepackage[margin=1in]{geometry}
\usepackage[T1]{fontenc}
\usepackage{lmodern}
\usepackage{amsmath,amssymb,amsthm,mathtools}
\usepackage{mathrsfs}
\usepackage{stmaryrd}
\usepackage{enumitem}
\usepackage{microtype}
\usepackage{hyperref}
\usepackage{bookmark}

\hypersetup{
  colorlinks=true,
  linkcolor=blue,
  citecolor=blue,
  urlcolor=blue,
  pdftitle={A strict gap above the Brito--Chacon--Naveira bound for the Sasaki volume on higher odd spheres},
  pdfauthor={Jonas Matuzas},
  pdfsubject={A strict smooth gap above the Brito--Chacon--Naveira lower bound},
  pdfkeywords={unit vector fields, Sasaki metric, Cartesian currents, calibrations, strict lower bound, odd spheres}
}

\allowdisplaybreaks
\setlist{itemsep=0.2em,topsep=0.4em}
\theoremstyle{plain}
\newtheorem{theorem}{Theorem}[section]
\newtheorem{proposition}[theorem]{Proposition}
\newtheorem{lemma}[theorem]{Lemma}
\newtheorem{corollary}[theorem]{Corollary}
\theoremstyle{definition}
\newtheorem{definition}[theorem]{Definition}
\newtheorem{example}[theorem]{Example}
\theoremstyle{remark}
\newtheorem{remark}[theorem]{Remark}
\newtheorem{problem}[theorem]{Problem}

\newcommand{\R}{\mathbb{R}}
\newcommand{\Z}{\mathbb{Z}}
\newcommand{\HH}{\mathcal{H}}
\newcommand{\vol}{\operatorname{vol}}
\newcommand{\dist}{\operatorname{dist}}
\newcommand{\M}{\mathbf{M}}
\newcommand{\Icur}{\mathbf{I}}

\newcommand{\Cm}{\mathcal{C}_m}
\newcommand{\gr}{\mathrm{gr}}
\newcommand{\res}{\mathrm{res}}
\newcommand{\VolS}{\operatorname{Vol}^{S}}

\newcommand{\Tan}{\operatorname{Tan}}

\newcommand{\Fflat}{\mathcal F}

\title{A strict gap above the Brito--Chac\'on--Naveira bound\\for the Sasaki volume on higher odd spheres}
\author{Jonas Matuzas\\[-0.15em]\normalsize\href{mailto:jonas.matuzas@gmail.com}{jonas.matuzas@gmail.com}}
\date{}

\begin{document}
\maketitle

\begin{abstract}
For a smooth unit vector field $V$ on the round sphere $S^{2m+1}$, its Sasaki
volume is the volume of its graph in the unit tangent bundle.  Brito, Chac\'on
and Naveira proved that this volume is at least
$c_m\operatorname{vol}(S^{2m+1})$, where
$c_m=4^m/\binom{2m}{m}$.  We prove that, for every $m\ge2$, the infimum over
smooth unit fields is strictly larger than this value.

The proof constructs a global closed comass-one form on $T^1S^{2m+1}$ and
classifies all of its equality planes.  A hypothetical sequence approaching
the bound has a calibrated integral-current limit.  Its positive
horizontal-Jacobian part is a multiplicity-one graph, whereas its
projection-degenerate residual is confined to a base-rank-one equality plane.
Coordinate test forms show that this residual contributes no singular part to
the distributional derivative of the graph section.  The section is therefore
Sobolev and satisfies a scalar concircular equation; a weak rigidity theorem
identifies it with a radial distance-gradient field.  The radial graph has the
two pole fibers as boundary, while a local boundary--mass identity forces the
residual to vanish, contradicting the cycle condition.

We also prove a sharp current-level flat-boundary inequality for fillings of
the pole fibers.  Consequently, every sequence of smooth unit fields
converging in measure to a radial field has lower limiting volume at least
twice the Brito--Chac\'on--Naveira bound.  The strict gap obtained here is
qualitative; no explicit gap constant is claimed.
\end{abstract}

{\small\noindent\textbf{MSC 2020:} 53C38, 49Q15 (primary); 49J45, 53C20 (secondary).\par}

\medskip
{\small\noindent\textbf{Generative-AI disclosure and author responsibility.}
This work was produced using OpenAI's ChatGPT 5.6 Pro.  The author directed
and audited the work throughout.  Jonas Matuzas takes full responsibility for
the mathematics and the final text.\par}

\tableofcontents

\section{Introduction and main results}\label{sec:intro}

Let $S^n\subset\R^{n+1}$ be the round unit sphere, with $n=2m+1$, and let
$E=T^1S^n$ carry the Sasaki metric.  A smooth unit field $V$ defines the
section $s_V(x)=(x,V(x))$ and the graph-volume functional
\[
\VolS(V)
=
\HH^n(s_V(S^n))
=
\int_{S^n}
\sqrt{\det\!\bigl(I+(\nabla V)^{\top}\nabla V\bigr)}\,d\vol.
\]
Gluck and Ziller proved that the Hopf fields minimize this functional on
$S^3$ \cite{GZ}; the corresponding uniqueness statement on positively curved
three-dimensional space forms was completed in \cite{GilM22}.  In every odd
dimension, Brito, Chac\'on and Naveira established
\begin{equation}\label{eq:bcn}
\VolS(V)\ge c_m\vol(S^{2m+1}),
\qquad
c_m=\frac{4^m}{\binom{2m}{m}}.
\end{equation}
Their pointwise equality model is the radial field
\[
R_p(x)=\frac{\cos r\,x-p}{\sin r},
\qquad r=\dist(x,p),
\]
which is singular at $p$ and $-p$.  They also proved that no smooth unit field
attains equality when $m\ge2$ \cite{BCN}.  Nonattainment, however, does not
exclude a sequence of smooth fields whose volumes converge to the bound.  The
question addressed here is whether such a sequence can exist.

\begin{theorem}[Strict gap above the BCN bound]\label{thm:main-intro}
For every $m\ge2$,
\[
\inf\left\{
\VolS(V):
V\in C^\infty(S^{2m+1};TS^{2m+1}),\ |V|\equiv1
\right\}
>
c_m\vol(S^{2m+1}).
\]
\end{theorem}

The proof is qualitative.  A compact relaxed class is used only to attain a
hypothetical equality sequence; the resulting equality current is shown not
to exist.  We deliberately call the conclusion a \emph{strict gap above the
BCN bound}, rather than a classical Lavrentiev gap: no second admissible
variational class attaining the BCN value is asserted.

A complementary theorem quantifies the radial basin.  The fiber orientations
used in its statement are fixed explicitly in Section~\ref{sec:prelim}.

\begin{theorem}[Flat-boundary floor and radial cost]\label{thm:floor-intro}
Let $m\ge1$, $E=T^1S^{2m+1}$, and let $\Sigma$ be a finite-mass
$(2m+1)$-current in $E$.  Put
\[
\beta=\Fflat\bigl(\partial\Sigma-(F_p+F_{-p})\bigr).
\]
Then
\[
\M(\Sigma)
\ge
\pi\vol(S^{2m})-\frac\pi2\,\beta.
\]
Consequently, if smooth unit fields $V_k$ converge in measure to $R_p$, then
\[
\liminf_{k\to\infty}\VolS(V_k)
\ge
2c_m\vol(S^{2m+1}).
\]
The parallel-transport cylinder realizes the finite-mass current floor; no
smooth recovery sequence at this value is asserted.
\end{theorem}

\subsection*{Proof architecture}
The proof of Theorem~\ref{thm:main-intro} has five steps.

\begin{enumerate}[label=\textup{(\roman*)}]
\item A globally defined closed form $\Omega=a\wedge\Theta$ has comass one.
For $m\ge2$, its equality cone consists of scalar graph planes and one
projection-degenerate endpoint plane of base rank one.
\item A hypothetical equality sequence has a calibrated integral-current limit
in the section homology class.
\item The limit splits into a multiplicity-one graph part and a
projection-null residual, and calibration pins the two pieces to the two
branches of the equality cone.
\item The rank-one residual is invisible to the coordinate tests computing the
distributional derivative of the graph section.  The section is therefore
Sobolev and satisfies the scalar calibrated equation.
\item Sobolev rigidity forces one radial field.  Its graph has the two pole
fibers as boundary.  A structure-equation identity then forces the residual
to vanish, contradicting the fact that the limit is a cycle.
\end{enumerate}

Sections~\ref{sec:calibration}--\ref{sec:proof} contain the strict-gap proof,
and Section~\ref{sec:floor} proves the quantitative radial statement.

\subsection*{Correction of the former no-gap claim}
The former no-gap preprint, arXiv:2602.22961, claimed that the BCN value was
the smooth infimum.  The proposed recovery construction did not control the full
high-dimensional graph Jacobian in the anti-alignment region.  No part of
that construction is used below.  Section~\ref{app:error} records the
correction.

\section{Geometric setting and current conventions}\label{sec:prelim}

\subsection{Sasaki geometry of \texorpdfstring{$T^1S^n$}{the unit tangent bundle}}
Throughout, $n=2m+1$ and $d=2m=n-1$.  We identify
\[
E=T^1S^n
 =\{(x,v)\in\R^{n+1}\times\R^{n+1}:
 |x|=|v|=1,\ x\cdot v=0\},
\]
a Stiefel manifold of dimension $2n-1=4m+1$.  If
$Z=(X,Y)\in T_{(x,v)}E$ and $K$ denotes the connection map, then the Sasaki
metric is $|Z|_{\rm S}^2=|X|^2+|KZ|^2$; it is not the metric induced by the
Euclidean product on the ambient space.  The projection is $\pi(x,v)=x$, and
its fibers are the unit spheres $S^d\subset T_xS^n$.

At $y=(x,v)\in E$, choose an orthonormal basis $e_1,\dots,e_d$ of
$\{x,v\}^\perp\subset\R^{n+1}$.  The \emph{Sasaki--Stiefel frame} of $T_yE$ is
\[
A=(v,-x),\qquad B_i=(e_i,0),\qquad C_i=(0,e_i),\qquad i=1,\dots,d.
\]
It is orthonormal for the Sasaki metric.  The vector $A$ generates the geodesic
flow, the $B_i$ span the complementary horizontal (contact) directions, and
the $C_i$ span the vertical directions.  In particular,
\[
d\pi(A)=v,\qquad d\pi(B_i)=e_i,\qquad d\pi(C_i)=0.
\]
Let $(a,b_i,c_i)$ be the dual coframe.  Standard moving-frame computations on
the Stiefel manifold (see, for example, \cite{GZ}, or \cite{Blair} for Sasaki
geometry) give the structure equations
\begin{equation}\label{eq:structure}
da=-\sum_i b_i\wedge c_i,\qquad
db_i=a\wedge c_i-\sum_j\eta_{ij}\wedge b_j,\qquad
dc_i=-a\wedge b_i-\sum_j\eta_{ij}\wedge c_j,
\end{equation}
with skew connection forms $\eta_{ij}=-\eta_{ji}$. We write
\[
\mathcal V:=\mathrm{span}\{A,C_1,\dots,C_d\}=\bigcap_{i=1}^d\ker b_i,
\qquad \operatorname{rank}\mathcal V=n.
\]
From \eqref{eq:structure}, $db_i(A,C_j)=(a\wedge c_i)(A,C_j)=\delta_{ij}$, so $b_i([A,C_j])=-\delta_{ij}$ and
\begin{equation}\label{eq:bracket}
[A,C_j]\equiv -B_j \pmod{\mathcal V}:
\end{equation}
$\mathcal V$ is nowhere involutive. (We will not need any Frobenius-type theorem; \eqref{eq:bracket} enters only through the algebraic identity of Lemma~\ref{lem:identity}.)

\subsection{Currents; graphs; conventions}
We use Federer's conventions \cite{Fed,Simon} for integral currents
$\Icur_k(E)$, mass $\M$, boundary $\partial S(\psi)=S(d\psi)$, restriction
$(S\llcorner\alpha)(\varphi)=S(\alpha\wedge\varphi)$ for a form $\alpha$, and
slicing.  The flat norm is
\[
\Fflat(Q):=\inf\bigl\{\M(R)+\M(S):Q=R+\partial S\bigr\},
\]
where $R$ has the same dimension as $Q$ and $S$ has one higher dimension.
All quantitative flat norms and flat convergences below are intrinsic to
$(E,g_{\rm S})$.  For a $1$-form $\omega$,
\begin{equation}\label{eq:restr}
\partial(S\llcorner\omega)=S\llcorner d\omega-(\partial S)\llcorner\omega,
\end{equation}
valid distributionally with no finiteness assumption on $\M(\partial S)$.

For a smooth unit field $V$, $T_V:=(s_V)_\#\llbracket S^n\rrbracket\in\Icur_n(E)$ satisfies $\partial T_V=0$, $\pi_\#T_V=\llbracket S^n\rrbracket$, and $\M(T_V)=\VolS(V)$.

\emph{Radial fields and pole fibers.}
Fix $p\in S^n$ and orient the Euclidean space $p^\perp$.  Let
$S_p^d\subset p^\perp$ be its oriented unit sphere and define
\[
F_p=(u\mapsto(p,u))_\#\llbracket S_p^d\rrbracket,
\qquad
F_{-p}=(u\mapsto(-p,u))_\#\llbracket S_p^d\rrbracket.
\]
Thus both fibers are parametrized by the same oriented sphere in the fixed
ambient space $p^\perp$.  This is an \emph{ambient-fixed} convention, not the
orientation obtained by parallel transport.  Indeed, parallel transport from
$p$ to $-p$ along a minimizing geodesic acts on $p^\perp$ as a reflection and
therefore has degree $-1$ on $S_p^d$; its endpoint fiber current is
$-F_{-p}$.

The radial field
\[
R_p(x)=\frac{\cos r\,x-p}{\sin r},
\qquad r=\dist(x,p)\in(0,\pi),
\]
is smooth off $\{\pm p\}$, satisfies $\nabla_{R_p}R_p=0$ and
$\nabla_XR_p=\cot r\,X$ for $X\perp R_p$, and defines a finite-mass graph
current $T_p$ by cutoff away from the poles and passage to the limit.

\begin{proposition}\label{prop:radial}
One has
\[
\M(T_p)=\pi\vol(S^{2m})=c_m\vol(S^{2m+1})
\]
and, in the ambient-fixed convention,
\begin{equation}\label{eq:radialbdry}
\partial T_p=-(F_p+F_{-p}).
\end{equation}
\end{proposition}
\begin{proof}
Write
\[
x=\cos r\,p+\sin r\,u,
\qquad
u\in S_p^d.
\]
Then
\[
R_p(x)=-\sin r\,p+\cos r\,u.
\]
The graph density is $\sin^{-d}r$, while the base volume element is
$\sin^dr\,dr\,d\sigma(u)$; hence
$\M(T_p)=\pi\vol(S^d)$.  The identity with
$c_m\vol(S^{2m+1})$ follows from Lemma~\ref{lem:comb} and the standard sphere
volume recursion.

For the boundary, orient $(0,\pi)\times S_p^d$ by
$dr\wedge d\sigma$.  On $r=\varepsilon$ the boundary orientation is
$-d\sigma$ and the trace tends to $u\mapsto(p,u)$, contributing $-F_p$.
On $r=\pi-\varepsilon$ the boundary orientation is $+d\sigma$ and the trace
tends to $u\mapsto(-p,-u)$.  The antipodal map on $S^d$, with $d=2m$ even,
has degree $(-1)^{d+1}=-1$, so this contribution is $-F_{-p}$.  Passing to
the limit gives \eqref{eq:radialbdry}.
\end{proof}

\emph{The parallel cylinder.}
Let $\gamma:[0,\pi]\to S^n$ be a minimizing geodesic from $p$ to $-p$ and
$P_t$ parallel transport along $\gamma$.  Define the smooth map
\[
\mathcal C_\gamma(t,u)=(\gamma(t),P_tu)
\]
and the current
\[
C_\gamma=(\mathcal C_\gamma)_\#
\llbracket[0,\pi]\times S_p^d\rrbracket.
\]
Then
\[
\partial C_\gamma=-F_{-p}-F_p,
\qquad
\M(C_\gamma)=\pi\vol(S^d).
\]
Thus $\Sigma_\gamma:=-C_\gamma$ fills $F_p+F_{-p}$.  This cylinder is used
only for orientation and sharpness bookkeeping.

\section{The calibration and its equality cone}\label{sec:calibration}

\begin{lemma}[Globality of the invariant forms]\label{lem:globalforms}
At $(x,v)\in E$, the ordered pair $(x,v)$ and the ambient orientation of
$\R^{n+1}$ orient the $d$-plane $\{x,v\}^\perp$.  Consequently two oriented
adapted frames differ by a matrix $Q\in SO(d)$.  Under the simultaneous frame
change
\[
(B_i,C_i)\longmapsto
\left(\sum_jQ_{ij}B_j,\sum_jQ_{ij}C_j\right),
\]
one has
\[
\bigwedge_{i=1}^d(c_i+t b_i)
\longmapsto
\det(Q)\bigwedge_{i=1}^d(c_i+t b_i)
=
\bigwedge_{i=1}^d(c_i+t b_i).
\]
Hence every coefficient $\tau_k$ below, and therefore $\Theta$ and
$\Omega=a\wedge\Theta$, is globally defined and smooth.
\end{lemma}

\begin{definition}\label{def:omega}
For $I\subset\{1,\dots,d\}$ let $\gamma^I_i=b_i$ if $i\in I$ and $\gamma^I_i=c_i$ if $i\notin I$, and set $\tau_k=\sum_{|I|=k}\gamma^I_1\wedge\cdots\wedge\gamma^I_d$. Define
\[
\Theta:=\sum_{j=0}^{m}C_{2j}\,\tau_{2j},\qquad
C_{2j}:=\frac{\binom mj}{\binom{2m}{2j}},\qquad
\Omega:=a\wedge\Theta.
\]
\end{definition}
At $(x,V(x))$, the tangent plane of a section graph is spanned by
\[
A+\sum_\beta b_\beta C_\beta,
\qquad
B_\alpha+\sum_\beta M_{\beta\alpha}C_\beta,
\]
where
$b_\beta=\langle\nabla_VV,e_\beta\rangle$ and
$M_{\beta\alpha}=\langle\nabla_{e_\alpha}V,e_\beta\rangle$.
If $\xi$ is its oriented unit tangent and $\rho$ its graph density, then
\[
\Omega(\xi)=\frac{\Phi_d(M)}{\rho},
\qquad
\Phi_d(M):=\sum_{j=0}^m C_{2j}\sigma_{2j}(M).
\]
Indeed, evaluating $\tau_{2j}$ produces the principal minors of order
$d-2j$.  The symmetry $C_{2j}=C_{d-2j}$ then gives the displayed indexing.  This is the
Brito--Chac\'on--Naveira integrand \cite{BCN}.

\begin{lemma}[Combinatorial identities]\label{lem:comb}
For all $m\ge1$: (i) $C_{2j}\binom{2m}{2j}=\binom mj$, hence $\Phi_d(\lambda I_d)=\sum_j\binom mj\lambda^{2j}=(1+\lambda^2)^m$; (ii) $\sum_{j=0}^m\binom mj^2/\binom{2m}{2j}=4^m/\binom{2m}{m}=c_m$; (iii) $C_{2j}$ equals the proportion of perfect matchings of $\{1,\dots,2m\}$ in which a fixed $2j$-element subset is matched internally.
\end{lemma}
\begin{proof}
(i) is immediate from the definition of $C_{2j}$. The key is the beta-integral representation
\begin{equation}\label{eq:beta}
C_{2j}=\frac1{I_0}\int_{-\pi/2}^{\pi/2}\cos^{2j}\varphi\,\sin^{d-2j}\varphi\,d\varphi,
\qquad I_0:=\int_{-\pi/2}^{\pi/2}\sin^{d}\varphi\,d\varphi=\pi\binom{2m}{m}4^{-m},
\end{equation}
where the value of $I_0$ is the Wallis integral for $d=2m$. Indeed, by the beta function,
\[
\int_{-\pi/2}^{\pi/2}\cos^{2j}\varphi\,\sin^{2m-2j}\varphi\,d\varphi
=\frac{\Gamma(j+\frac12)\,\Gamma(m-j+\frac12)}{\Gamma(m+1)}
=\pi\,\frac{(2j)!\,(2m-2j)!}{4^m\,j!\,(m-j)!\,m!},
\]
and dividing by $I_0$ gives $\binom mj\frac{(2j)!(2m-2j)!}{(2m)!}=\binom mj\big/\binom{2m}{2j}=C_{2j}$. For (ii), sum \eqref{eq:beta} against $\binom mj$ and use the binomial theorem:
\[
\sum_{j=0}^m\frac{\binom mj^2}{\binom{2m}{2j}}=\sum_j\binom mjC_{2j}
=\frac1{I_0}\int_{-\pi/2}^{\pi/2}\big(\cos^2\varphi+\sin^2\varphi\big)^{m}d\varphi
=\frac{\pi}{I_0}=\frac{4^m}{\binom{2m}{m}}.
\]
(iii): perfect matchings of $2m$ labels number $(2m)!/(2^mm!)$; those matching a fixed $2j$-element set internally number $\frac{(2j)!}{2^jj!}\cdot\frac{(2m-2j)!}{2^{m-j}(m-j)!}$; the ratio simplifies to $\binom mj/\binom{2m}{2j}$.
\end{proof}

\begin{theorem}[Closedness]\label{thm:closed}
$d\Theta=0$ and $d\Omega=0$.
\end{theorem}
\begin{proof}
Let $\mathscr P(t)=\bigwedge_{i=1}^d(c_i+tb_i)=\sum_k t^k\tau_k$. By \eqref{eq:structure}, $d(c_i+tb_i)=a\wedge(tc_i-b_i)-\sum_j\eta_{ij}\wedge(c_j+tb_j)$. In $d\mathscr P(t)$ the connection terms cancel: the $\eta_{ij}$-term of factor $i$ produces $\pm\eta_{ij}\wedge(c_j+tb_j)\wedge\prod_{k\ne i}(c_k+tb_k)$, which vanishes when $j\ne i$ by duplication of the factor $(c_j+tb_j)$, and $\eta_{ii}=0$. Hence
\[
d\tau_k=a\wedge\big((d-k+1)\tau_{k-1}-(k+1)\tau_{k+1}\big),
\]
so $d\Theta=a\wedge\sum_{j=0}^{m-1}\big[(2m-2j-1)\,C_{2j+2}-(2j+1)\,C_{2j}\big]\tau_{2j+1}$ (only odd indices survive since $\Theta$ has even ones). Every bracket vanishes by the coefficient identity
\[
\frac{C_{2j+2}}{C_{2j}}
=\frac{\binom{m}{j+1}}{\binom mj}\cdot\frac{\binom{2m}{2j}}{\binom{2m}{2j+2}}
=\frac{m-j}{j+1}\cdot\frac{(2j+1)(2j+2)}{(2m-2j)(2m-2j-1)}
=\frac{2j+1}{2m-2j-1},
\]
so $d\Theta=0$. Moreover every monomial of $\Theta$ contains, for each index $i$, one of $b_i,c_i$; hence $(b_i\wedge c_i)\wedge\Theta=0$ for every $i$ and $da\wedge\Theta=-\sum_i b_i\wedge c_i\wedge\Theta=0$. Therefore $d\Omega=da\wedge\Theta-a\wedge d\Theta=0$: both terms vanish separately.
\end{proof}
\begin{remark}
That $\Theta$ itself is closed --- not merely $a\wedge d\Theta=0$ --- is a genuinely stronger statement: the horizontal correction terms $\tau_{2j}$, $j\ge1$, exactly close the non-closed vertical form $\tau_0=c_1\wedge\cdots\wedge c_d$ (for which $d\tau_0=-a\wedge\tau_1\ne0$). Closed forms of comass one in the fiber direction are the natural dual objects for lower bounds on $\pi$-degenerate currents.  The form $\Theta$ will be used in Section~\ref{sec:floor} to price pole-fiber boundaries.
\end{remark}

\begin{theorem}[Full comass and the equality cone]\label{thm:comass}
For every $m\ge1$, $\|\Omega\|_*=1$: for every oriented unit simple
$n$-vector $\xi$ of $T_yE$, $|\Omega(\xi)|\le1$.  If $m\ge2$, then
$\Omega(\xi)=1$ holds exactly for
\[
\xi=A\wedge\eta_\lambda\quad(\lambda\in\R),
\qquad
\eta_\lambda=(1+\lambda^2)^{-m}
\bigwedge_{i=1}^d(B_i+\lambda C_i),
\]
up to an oriented simultaneous rotation of the $B$- and $C$-frames, or for
\[
\xi=A\wedge C_1\wedge\cdots\wedge C_d.
\]
In particular, for $m\ge2$ the only calibrated plane with $J_n\pi=0$ is the
last one, whose base projection has rank one.
\end{theorem}
\begin{proof}
Since $\Omega=a\wedge\Theta$ and $|a(\xi_1)|\le1$, it suffices to prove $|\Theta(P)|\le1$ for oriented unit $d$-planes $P$ in the contact hyperplane $\mathscr D=\ker a=\mathrm{span}\{B_i,C_i\}$, with equality classification; a general $\xi$ then satisfies $|\Omega(\xi)|\le\|\iota_A\xi\|\le1$ with equality forcing $A\in\xi$ and the contact factor calibrated.

Choose an orthonormal basis $f_1,\dots,f_d$ of $P$ and write $f_i=b_i^H+c_i^V$ with $b_i,c_i\in\R^d$ (identifying the $B$- and $C$-blocks with $\R^d$), so $|b_i|^2+|c_i|^2=1$. Let $\mathbf B,\mathbf C$ be the $d\times d$ matrices with rows $b_i,c_i$. Evaluating the coframe monomials, $\tau_k(P)=\sum_{|I|=k}\det M_I$, where $M_I$ has column $j$ taken from $\mathbf B$ if $j\in I$ and from $\mathbf C$ otherwise. Expanding the determinant below by columns and integrating (odd powers of $\sin$ vanish by parity; even powers produce the coefficients \eqref{eq:beta}, using $C_{2j}=C_{2(m-j)}$) yields the \emph{rotation-integral identity}
\begin{equation}\label{eq:rotint}
\Theta(P)=\frac1{I_0}\int_{-\pi/2}^{\pi/2}\det\big(\cos\varphi\,\mathbf C+\sin\varphi\,\mathbf B\big)\,d\varphi.
\end{equation}
Bound the determinant by Hadamard's inequality over rows, $|\det|\le\prod_i\|\cos\varphi\,c_i+\sin\varphi\,b_i\|$, and compute each row norm exactly:
\begin{equation}\label{eq:rownorm}
\|\cos\varphi\,c_i+\sin\varphi\,b_i\|^2
=\tfrac12+\tfrac12 R_i\cos(2\varphi-\delta_i),
\qquad
R_i:=\sqrt{\big(|c_i|^2-|b_i|^2\big)^2+4\langle b_i,c_i\rangle^2}\ \le\ 1,
\end{equation}
where the bound $R_i\le1$ is Cauchy--Schwarz: $R_i^2\le(|c_i|^2-|b_i|^2)^2+4|b_i|^2|c_i|^2=(|b_i|^2+|c_i|^2)^2=1$. By AM--GM on $d$-th powers, $\prod_i\|\text{row}_i\|\le\frac1d\sum_i\|\text{row}_i\|^{d}$, and each summand integrates to the \emph{moment}
\[
G(R):=\int_{-\pi/2}^{\pi/2}\Big(\tfrac12+\tfrac12R\cos(2\varphi-\delta)\Big)^{m}d\varphi
=\sum_{k=0}^{\lfloor m/2\rfloor}\binom m{2k}2^{-m}\,R^{2k}\int_{-\pi/2}^{\pi/2}\cos^{2k}\psi\,d\psi ,
\]
independent of the phase $\delta$ ($\psi=2\varphi-\delta$ runs over a full period of the integrand) and \emph{nondecreasing in $R\in[0,1]$}, with
\[
G(1)=\tfrac12\int_0^{2\pi}\cos^{2m}(\psi/2)\,d\psi=\int_0^{\pi}\cos^{2m}u\,du=I_0.
\]
Hence $|\Theta(P)|\le\frac1{I_0}\cdot\frac1d\sum_iG(R_i)\le\frac1{I_0}G(1)=1$, proving the comass bound.

\emph{Equality, $m\ge2$.}
For $m\ge2$, the $R^2$-coefficient of $G$ is positive, so $G$ is strictly
increasing on $[0,1]$.  Equality in the preceding chain therefore forces,
for every $i$ and for a.e.\ $\varphi$,
\begin{enumerate}[label=\textup{(E\arabic*)},leftmargin=3.2em]
\item $R_i=1$.  Equivalently, $b_i$ and $c_i$ are parallel, with the cases in
which one vanishes included.
\item Equality holds in AM--GM\@.  Thus the row norms agree, and
\eqref{eq:rownorm} gives $\delta_i\equiv\delta\pmod{2\pi}$.
\item Equality holds in Hadamard's inequality.  Hence the rows are pairwise
orthogonal for a.e.\ $\varphi$.  Expanding their inner products in
$\{1,\cos2\varphi,\sin2\varphi\}$ gives, for $i\ne j$,
\[
\langle b_i,b_j\rangle=\langle c_i,c_j\rangle=0,
\qquad
\langle b_i,c_j\rangle+\langle c_i,b_j\rangle=0.
\]
\item Equality holds in
$\bigl|\int\det\bigr|\le\int|\det|$, so the determinant has a fixed sign
outside its zero set.
\end{enumerate}

By (E1), write
\[
b_i=\cos\alpha_i\,u_i,
\qquad
c_i=\varepsilon_i\sin\alpha_i\,u_i,
\]
where $u_i$ is unit, $\alpha_i\in[0,\pi/2]$, and
$\varepsilon_i\in\{\pm1\}$.  The phase of row $i$ then satisfies
\[
\frac{\delta_i}{2}\equiv
\varepsilon_i\Bigl(\frac\pi2-\alpha_i\Bigr)\pmod\pi.
\]
Its representative is $0$ in the vertical case, $\pi/2$ in the horizontal
case, and otherwise belongs to
$(-\pi/2,\pi/2)\setminus\{0\}$.  Condition (E2) therefore forces either all
rows to be vertical, giving $P=C_1\wedge\cdots\wedge C_d$, or all rows to be
nonvertical with a common pair $(\alpha,\varepsilon)$.  In the latter case,
(E3) makes $(u_i)$ an orthonormal basis of $\R^d$, and the underlying plane is
\[
P=\{(w,\lambda w):w\in\R^d\},
\qquad
\lambda=\varepsilon\tan\alpha\in\R.
\]
This is the scalar graph plane $\eta_\lambda$; $\lambda=0$ is the horizontal
case.  Finally, (E4) fixes the orientation: the displayed graph and vertical
orientations give $\Theta=1$, while their reversals give $\Theta=-1$.

\emph{The case $m=1$.}
The comass estimate remains valid, but $G(R)$ is constant and the strict
step used above disappears.  The equality cone is therefore larger.  Its
full classification is not needed in this paper; this loss of rigidity is
consistent with the Hopf equality fields on $S^3$.

Finally, the base projections: $d\pi(A\wedge\eta_\lambda)$ has full rank with Jacobian $(1+\lambda^2)^{-m}>0$ for every $\lambda\in\R$ (\S\ref{sec:decomp}), while $d\pi(A\wedge C_1\wedge\cdots\wedge C_d)$ has rank one ($d\pi A=v$, $d\pi C_i=0$).
\end{proof}

\begin{remark}
By Lemma~\ref{lem:comb}(i), on graph planes with block $\lambda I$, $\Theta=(1+\lambda^2)^m/(1+\lambda^2)^{m}=1$: the radial fields are calibrated at every regular point, for every $\lambda=\cot r$. For the Hopf block, $M$ is an orthogonal complex structure $J$ and $\det(I+tJ)=(1+t^2)^m$, so $\sigma_{2j}(J)=\binom mj$ and $\Phi_d(J)=\sum_j C_{2j}\binom mj=c_m$.  Its graph density is $2^m$, hence the pointwise ratio is $c_m/2^m<1$ for $m\ge2$; Hopf fields cease to be calibrated above $S^3$.
\end{remark}

\begin{proposition}[Charge identity]\label{prop:charge}
For every smooth unit $V$, $T_V(\Omega)=c_m\vol(S^{2m+1})$, and this passes to weak limits.
\end{proposition}
\begin{proof}
Since $n$ is odd, the Euler class of $TS^n$ vanishes.  The Gysin sequence for the oriented sphere bundle $S^d\hookrightarrow E\to S^n$ (see, for example, \cite{BT}) gives $H_n(E;\Z)\cong H_n(S^n;\Z)=\Z$, with $\pi_\#$ an isomorphism in degree $n$. Any two section currents push forward to the generator $\llbracket S^n\rrbracket$, hence are homologous in $E$; realizing the homology by an integral $(n{+}1)$-current $W$ (Federer--Fleming) and using $d\Omega=0$ (Theorem~\ref{thm:closed}) gives $T_{V_0}(\Omega)-T_{V_1}(\Omega)=\partial W(\Omega)=W(d\Omega)=0$. To evaluate the common value, note $s_V^*\Omega=\Phi_d(\nabla V)\,d\vol$; for a Hopf field $V$ on $S^{2m+1}$ the shape block is an orthogonal complex structure $J$, with $\sigma_{2j}(J)=\binom mj$ and $\sigma_{\mathrm{odd}}(J)=0$ (from $\det(I+tJ)=(1+t^2)^m$), so $\Phi_d(J)=\sum_jC_{2j}\binom mj=c_m$ by Lemma~\ref{lem:comb}(ii), a constant: $T_V(\Omega)=c_m\vol(S^n)$. Weak convergence preserves the pairing with the fixed smooth form $\Omega$.
\end{proof}

\section{Relaxation and calibrated minimizers}\label{sec:framework}

\begin{definition}
For $\Lambda<\infty$, let $\mathcal G_m(\Lambda)=\{T_V:\ V\ \text{smooth unit},\ \M(T_V)\le\Lambda\}$ and $\Cm(\Lambda)$ its weak closure in $\Icur_n(E)$. Fix once and for all the \emph{Hopf cap} $\Lambda_H:=2^m\vol(S^{2m+1})+1$.
\end{definition}

\begin{lemma}[Sasaki versus ambient mass]\label{lem:sasaki-ambient}
View
\[
E=T^1S^n=\{(x,v)\in\R^{n+1}\times\R^{n+1}:
|x|=|v|=1,\ x\cdot v=0\}
\]
with both the Sasaki metric $g_{\rm S}$ and the metric $g_{\rm amb}$
induced by the Euclidean product.  If
$Z=(X,Y)\in T_{(x,v)}E$ and
\[
KZ:=Y+\langle X,v\rangle x,
\]
then
\[
|Z|_{\rm S}^2=|X|^2+|KZ|^2,\qquad
|Z|_{\rm amb}^2=|X|^2+|KZ|^2+\langle X,v\rangle^2.
\]
Consequently
\[
|Z|_{\rm S}\le |Z|_{\rm amb}\le\sqrt2\,|Z|_{\rm S}.
\]
For every integer-rectifiable $k$-current $S$ supported in $E$,
\[
\M_{\rm S}(S)\le \M_{\rm amb}(S)
\le 2^{k/2}\M_{\rm S}(S).
\]
\end{lemma}

\begin{proof}
The tangent constraints are
$x\cdot X=v\cdot Y=X\cdot v+x\cdot Y=0$.
Hence $KZ\perp x,v$ and
$Y=KZ-\langle X,v\rangle x$, which gives the two displayed norm
identities.  Since
$|\langle X,v\rangle|\le|X|$, the pointwise comparison follows.
Taking exterior powers gives
$|\xi|_{\rm S}\le|\xi|_{\rm amb}\le2^{k/2}|\xi|_{\rm S}$ for every
simple $k$-vector $\xi$, and integration gives the mass comparison.
\end{proof}

\begin{proposition}[Compactness, charge, and attainment]\label{prop:compactness}
For every finite $\Lambda$, the class $\Cm(\Lambda)$ is compact in the intrinsic flat topology of $E$.  Every $T\in\Cm(\Lambda)$ is an integral cycle satisfying
\[
\pi_\#T=\llbracket S^n\rrbracket,
\qquad
T(\Omega)=c_m\vol(S^n),
\qquad
\M(T)\le\Lambda.
\]
Consequently
\[
m_{\mathrm{rel}}
:=
\min_{T\in\Cm(\Lambda_H)}\M(T)
\]
is attained.  We write
\[
\mathcal K_m
:=
\left\{
T\in\Cm(\Lambda_H):
\M(T)=T(\Omega)=c_m\vol(S^n)
\right\}
\]
for the calibrated floor class.
\end{proposition}
\begin{proof}
Let $T$ be a weak limit of graph currents $T_{V_j}$ with
$\M_{\rm S}(T_{V_j})\le\Lambda$.  The Federer--Fleming compactness theorem,
applied on the compact Riemannian manifold $(E,g_{\rm S})$ through a finite
atlas \cite[\S4.2.17]{Fed}, gives an intrinsically flat-convergent subsequence
with integral limit.  The same subsequence converges weakly to $T$, so its
integral limit is $T$; in particular, $\partial T=0$.  Continuity of
push-forward and Proposition~\ref{prop:charge} give the projection and charge
identities, while lower semicontinuity gives $\M(T)\le\Lambda$.

Now let $(T_j)\subset\Cm(\Lambda)$.  The same compactness theorem gives an
intrinsically flat-convergent subsequence.  Its weak limit belongs to
$\Cm(\Lambda)$ because that class is weakly closed by definition.  Thus
$\Cm(\Lambda)$ is flat compact.  Lower semicontinuity of mass then
gives attainment of $m_{\mathrm{rel}}$.  The comass inequality gives
$m_{\mathrm{rel}}\ge c_m\vol(S^n)$, so the displayed definition of
$\mathcal K_m$ is exactly the equality class.
\end{proof}

\section{Structure of calibrated limits}\label{sec:decomp}

Let $dV$ be the oriented volume form of $S^n$.  If
$T=\llbracket M,\theta,\xi\rrbracket\in\Icur_n(E)$, define the signed
horizontal Jacobian
\[
j_\pi(z):=\langle\pi^*dV,\xi(z)\rangle.
\]

\begin{theorem}[Positive graph--residual decomposition]\label{thm:mcs}
Let $U_j$ be smooth unit vector fields on $S^n$ such that
\[
\sup_j\M_{\rm S}(T_{U_j})<\infty,
\qquad
T_{U_j}\rightharpoonup T.
\]
Then $T$ is an integral cycle, $\pi_\#T=\llbracket S^n\rrbracket$, and
the following hold.

\begin{enumerate}[label=\textup{(\roman*)}]
\item
$j_\pi\ge0$ $\|T\|$-a.e.  Put
\[
M_{\rm gr}:=\{j_\pi>0\},\qquad M_{\rm res}:=\{j_\pi=0\},
\]
and
\[
T_{\rm gr}:=T\llcorner M_{\rm gr},\qquad
T_{\rm res}:=T\llcorner M_{\rm res}.
\]
These are integer-rectifiable finite-mass currents with mutually
singular mass measures and
\[
T=T_{\rm gr}+T_{\rm res},\qquad
\M(T)=\M(T_{\rm gr})+\M(T_{\rm res}).
\]
No finite-boundary-mass assertion is made for the two pieces.

\item
$J_n\pi=0$ $\|T_{\rm res}\|$-a.e.\ and
\[
\pi_\#(T_{\rm res}\llcorner f)=0
\]
for every bounded Borel function $f$.

\item
Up to an $\HH^n$-null set, $M_{\rm gr}$ is covered by countably many
$C^1$ graphs over open subsets of $S^n$.  There is a Borel unit section
$U:S^n\to TS^n$, unique a.e., such that
\[
M_{\rm gr}=\{(x,U(x)):x\in S^n\}
\quad\text{up to }\HH^n\text{-null sets},
\]
the multiplicity is one, and the orientation is the positive graph
orientation.  Equivalently, for a.e. $x$,
\[
\langle T_{\rm gr},\pi,x\rangle
=
\llbracket(x,U(x))\rrbracket.
\]

\item
The map $U:S^n\to\R^{n+1}$ is approximately differentiable a.e.\ and
\[
\int_{S^n}|D^{\rm ap}U|\,dV
\le
\M_{\rm amb}(T_{\rm gr})
\le
2^{n/2}\M_{\rm S}(T_{\rm gr}).
\]

\item
Let $B$ be an oriented coordinate chart, let
$y^\alpha(x,v)=v^\alpha$, and set
\[
\widehat{dx^i}
=
(-1)^{i-1}dx^1\wedge\cdots\wedge
\widehat{dx^i}\wedge\cdots\wedge dx^n.
\]
For
\[
\omega_{i\alpha}(\varphi)
=
\varphi(x)y^\alpha\widehat{dx^i},
\qquad
\varphi\in C_c^\infty(B),
\]
one has
\[
T_{\rm gr}\!\left(d\omega_{i\alpha}(\varphi)\right)
=
\int_B
\left(
U^\alpha\partial_i\varphi
+
\varphi D_i^{\rm ap}U^\alpha
\right)\,dx.
\]

\item
For every $i,\alpha$ the functional
\[
\mu_i^\alpha(\varphi)
:=
T_{\rm res}
\left(
\varphi\,dy^\alpha\wedge\widehat{dx^i}
\right)
\]
is a finite signed Radon measure on compact subsets of $B$, and
\[
D_iU^\alpha
=
D_i^{\rm ap}U^\alpha\,\mathcal L^n
+
\mu_i^\alpha
\quad\text{in }\mathcal D'(B).
\]
In particular, $U\in BV_{\rm loc}(B;\R^{n+1})$.  More precisely, for
$K\Subset B$,
\[
|\mu_i^\alpha|(K)
\le
\sup_{\pi^{-1}K}
\left|
dy^\alpha\wedge\widehat{dx^i}
\right|_{\rm S}\,
\|T_{\rm res}\|(\pi^{-1}K).
\]

\item
If
\[
T_{\rm res}\!\left(d\omega_{i\alpha}(\varphi)\right)=0
\quad
\text{for every }i,\alpha,\varphi,
\]
then $\mu_i^\alpha=0$ for all $i,\alpha$ and
$U\in W^{1,1}_{\rm loc}(B;\R^{n+1})$, with weak derivative
$D^{\rm ap}U$.
\end{enumerate}
\end{theorem}

\begin{proof}
The compactness statement and the cycle equation follow from
Proposition~\ref{prop:compactness}.  We prove the structural assertions.

\emph{Step 1: positivity and the two strata.}
For every nonnegative $\psi\in C^\infty(E)$,
\[
T_{U_j}(\psi\,\pi^*dV)
=
\int_{S^n}\psi(x,U_j(x))\,dV(x)\ge0.
\]
Passing to the weak limit gives
\[
\int_M\psi\,\theta\,j_\pi\,d\HH^n\ge0
\]
for all nonnegative $\psi$, hence $j_\pi\ge0$ $\|T\|$-a.e.
The two strata are Borel, because the approximate tangent and its
orientation are Borel on a rectifiable carrier.  Restriction to them
gives integer-rectifiable finite-mass currents; mutual singularity gives
mass additivity.  On $M_{\rm res}$ the $n$-Jacobian of $\pi$ vanishes, so
for every bounded Borel $f$ and every smooth compactly supported
$n$-form $\eta$ on the base,
\[
\pi_\#(T_{\rm res}\llcorner f)(\eta)
=
\int_{M_{\rm res}}
f\,\theta\,\langle\pi^*\eta,\xi\rangle\,d\HH^n
=
0.
\]

\emph{Step 2: graph patches and multiplicity one.}
A countably rectifiable set is covered, up to a null set, by countably
many $C^1$ submanifolds.  At every point of $M_{\rm gr}$ the differential
of $\pi$ on the tangent plane is an orientation-preserving isomorphism.
The inverse-function theorem on each covering submanifold therefore
gives a countable cover of $M_{\rm gr}$ by $C^1$ graphs over the base.

Apply the area formula to $\pi|_{M_{\rm gr}}$.  For a.e. $x$ put
\[
N(x)
:=
\sum_{z\in M_{\rm gr}\cap\pi^{-1}(x)}\theta(z).
\]
Because $j_\pi>0$ on $M_{\rm gr}$,
\[
\pi_\#T_{\rm gr}=N\,\llbracket S^n\rrbracket.
\]
On the other hand,
\[
\pi_\#T_{\rm gr}
=
\pi_\#T-\pi_\#T_{\rm res}
=
\llbracket S^n\rrbracket.
\]
Hence $N(x)=1$ for a.e. $x$.  Since every summand is a positive integer,
there is exactly one point $(x,U(x))$ in the fiber and its multiplicity
is one.  The countable graph representation makes $U$ Borel and
approximately differentiable a.e.; its approximate graph tangent is
$\operatorname{graph}(D^{\rm ap}U)$.

\emph{Step 3: derivative and mass control.}
For the ambient graph map $F(x)=(x,U(x))$,
\[
J_n^{\rm amb}F
=
\sqrt{\det\!\left(I+(D^{\rm ap}U)^*D^{\rm ap}U\right)}.
\]
If $\sigma_1,\dots,\sigma_n$ are the singular values of
$D^{\rm ap}U$, then
\[
\prod_{i=1}^n(1+\sigma_i^2)
\ge
1+\sum_{i=1}^n\sigma_i^2.
\]
Thus $J_n^{\rm amb}F\ge|D^{\rm ap}U|$.  Integrating over the
multiplicity-one graph and using Lemma~\ref{lem:sasaki-ambient} proves
item \textup{(iv)}.

\emph{Step 4: the coordinate formula.}
On every $C^1$ graph patch, pullback by $F$ gives, at a.e.\ point,
\[
F^*d\omega_{i\alpha}(\varphi)
=
\left(
U^\alpha\partial_i\varphi
+
\varphi D_i^{\rm ap}U^\alpha
\right)
dx^1\wedge\cdots\wedge dx^n.
\]
The integrand is in $L^1$ by item \textup{(iv)}.  Summing the patches and
using multiplicity one proves item \textup{(v)}.

\emph{Step 5: the singular derivative is carried by the residual.}
Because $T$ is a cycle,
\[
0=T\!\left(d\omega_{i\alpha}(\varphi)\right).
\]
The term
$\partial_i\varphi\,y^\alpha\,dx^1\wedge\cdots\wedge dx^n$
vanishes on $T_{\rm res}$, since its tangent has zero horizontal
$n$-Jacobian.  Therefore
\[
0=
\int_B
\left(
U^\alpha\partial_i\varphi
+
\varphi D_i^{\rm ap}U^\alpha
\right)\,dx
+
\mu_i^\alpha(\varphi).
\]
The mass estimate in item \textup{(vi)} shows that $\mu_i^\alpha$ is a
Radon measure.  Rewriting the last display in the definition of the
distributional derivative gives
\[
D_iU^\alpha
=
D_i^{\rm ap}U^\alpha\mathcal L^n+\mu_i^\alpha.
\]
This proves item \textup{(vi)}.  Finally, if the residual annihilates all
$d\omega_{i\alpha}(\varphi)$, its horizontal-$n$ term already vanishes,
so $\mu_i^\alpha=0$ and item \textup{(vii)} follows.
\end{proof}

\begin{remark}[Scope of the positivity]\label{rem:jpi-scope}
On the smooth graph sector the positivity of $j_\pi$ carries no selective information: for \emph{every} unit section $U$, the projection $\pi$ restricted to $\operatorname{graph}(U)$ is an orientation-compatible diffeomorphism onto $S^n$, so the area formula gives $j_\pi\,d\|T_U\|=dV$ identically, radial or not. The content of the theorem lies in what it forbids of \emph{limits}: negative horizontal sheets and multiplicity in the graph part. Quantitative use of the decomposition must therefore act through the residual, never through $j_\pi$ alone.
\end{remark}

\begin{lemma}[Calibration of the two strata]\label{lem:piececal}
If $T\in\mathcal K_m$, then
\[
\M(T_{\rm gr})=T_{\rm gr}(\Omega),
\qquad
\M(T_{\rm res})=T_{\rm res}(\Omega).
\]
\end{lemma}
\begin{proof}
The decomposition has mutually singular mass measures, so
$\M(T)=\M(T_{\rm gr})+\M(T_{\rm res})$.  The comass inequality gives
$T_{\rm gr}(\Omega)\le\M(T_{\rm gr})$ and
$T_{\rm res}(\Omega)\le\M(T_{\rm res})$.  Their sums are equal because
$T$ is calibrated.  Both inequalities are therefore equalities.
\end{proof}

For a calibrated $T\in\mathcal K_m$, Lemma~\ref{lem:piececal},
Theorem~\ref{thm:mcs}, and Theorem~\ref{thm:comass} now give:
\begin{equation}\label{eq:restangent}
\xi_{T_\res}=A\wedge C_1\wedge\cdots\wedge C_d\quad\|T_\res\|\text{-a.e.},\qquad
\xi_{T_\gr}=A\wedge\eta_{\lambda(x)}\quad\|T_\gr\|\text{-a.e.}
\end{equation}
(both strata are calibrated a.e.; the residual stratum has $J_n\pi=0$, and the only such calibrated plane is $A\wedge\mathrm{Vert}$; the graph stratum has $j_\pi>0$, forcing the scalar branch --- note $\pi_*(A\wedge\eta_\lambda)=(1+\lambda^2)^{-m}\,v\wedge e_1\wedge\cdots\wedge e_d$ is \emph{positively} oriented for every real $\lambda$, and a reversed sheet would have $\Omega(\xi)=-1$, excluded by calibratedness).

\section{A boundary--mass identity}\label{sec:identity}

\begin{lemma}[Local boundary--mass identity]\label{lem:identity}
Let $T\in\mathcal K_m$, $T=T_\gr+T_\res$.  On every open set $O\Subset E$ carrying a smooth adapted frame $(A,B_i,C_i)$, and for every $i=1,\dots,d$,
\begin{equation}\label{eq:massid}
T_\res\llcorner(a\wedge c_i)=(\partial T_\res)\llcorner b_i=-(\partial T_\gr)\llcorner b_i
\quad\hbox{in }O,
\end{equation}
and
\[
\|T_\res\|\llcorner O=\big\| (\partial T_\gr)\llcorner b_i\big\|\llcorner O.
\]
No finiteness assumption on the full masses $\M(\partial T_\res)$ or $\M(\partial T_\gr)$ is required.  Consequently, if $\partial T_\gr$ is vertical, then $T_\res=0$.
\end{lemma}
\begin{proof}
The statement is local because the coframe is local.  On $O$, $b_i$ annihilates $\mathcal V\supset\Tan T_\res$, so $T_\res\llcorner b_i=0$.  Formula~\eqref{eq:restr} with $\omega=b_i$ gives $T_\res\llcorner db_i=(\partial T_\res)\llcorner b_i$.  Using \eqref{eq:structure}, all connection terms vanish on $T_\res$ and therefore $T_\res\llcorner(a\wedge c_i)=(\partial T_\res)\llcorner b_i$.  Since $\partial T=0$, the latter equals $-(\partial T_\gr)\llcorner b_i$.  Finally,
\[
(a\wedge c_i)\lrcorner(A\wedge C_1\wedge\cdots\wedge C_d)
=\pm C_1\wedge\cdots\widehat{C_i}\cdots\wedge C_d
\]
has unit norm, so restriction by $a\wedge c_i$ preserves the residual mass measure on $O$.  If $\partial T_\gr$ is vertical, every $b_i$ annihilates it locally; a finite adapted-frame cover of the compact carrier then gives $\|T_\res\|=0$.
\end{proof}

Thus the residual vanishes \emph{if and only if} the graph part's boundary carries no $b_i$-visible trace. The next section proves that, for calibrated limits, it cannot.

\section{Residual invisibility and Sobolev regularity}\label{sec:key}

Fix a base chart with coordinates $x=(x^1,\dots,x^n)$ and use \emph{ambient fiber components}: $y^\alpha(x,v)=v^\alpha$, $\alpha=1,\dots,n+1$, restricted to $E$. Set $\widehat{dx^i}=(-1)^{i-1}dx^1\wedge\cdots\wedge\widehat{dx^i}\wedge\cdots\wedge dx^n$, so $dx^i\wedge\widehat{dx^i}=dx^1\wedge\cdots\wedge dx^n=:dx$. For $\varphi\in C_c^\infty$ of the chart define the \emph{coordinate tests}
\[
\omega_{i\alpha}(\varphi):=\varphi(x)\,y^\alpha\,\widehat{dx^i},\qquad
d\omega_{i\alpha}=\partial_i\varphi\,y^\alpha\,dx+\varphi\,dy^\alpha\wedge\widehat{dx^i}.
\]

\begin{lemma}[Residual invisibility]\label{lem:key}
Let $T\in\mathcal K_m$, $m\ge2$ (the computation itself needs only $n-1\ge2$). Then for every $\varphi,i,\alpha$:
\[
T_\res\big(d\omega_{i\alpha}(\varphi)\big)=0.
\]
\end{lemma}
\begin{proof}
By \eqref{eq:restangent} the tangent is $A\wedge C_1\wedge\cdots\wedge C_d$ a.e., with base projections $\pi_*A=v$, $\pi_*C_r=0$: rank one. The first term of $d\omega_{i\alpha}$ has $n$ base differentials; evaluated on any $n$-tuple from $\{A,C_1,\dots,C_d\}$ it vanishes, since at most one argument has nonzero base projection. For the second term, expand by which slot takes $dy^\alpha$:
\[
\begin{aligned}
(dy^\alpha\wedge\widehat{dx^i})(A,C_1,\dots,C_d)
&=dy^\alpha(A)\,\widehat{dx^i}(C_1,\dots,C_d)\\
&\quad+\sum_{r}(-1)^r dy^\alpha(C_r)\,\widehat{dx^i}(A,C_1,\dots,\widehat{C_r},\dots,C_d).
\end{aligned}
\]
Here $dy^\alpha(A)=-x^\alpha$ is generally nonzero --- but $\widehat{dx^i}(C_1,\dots,C_d)=0$ since every $C_r$ has zero base projection; and in each remaining summand the $(n-1)$-fold base form is evaluated on arguments of which only $A$ has nonzero base projection, so it vanishes because $n-1\ge2$. The integrand $\langle d\omega_{i\alpha},\xi\rangle$ therefore vanishes $\|T_\res\|$-a.e.; it is bounded, and $T_\res$ has finite mass, so the evaluation is zero.
\end{proof}

\begin{example}[The wall discriminant: invisibility uses calibration]\label{ex:wall}
Plain $\pi$-nullity would allow residual base rank $n-1$, and then the lemma fails. Let $J=\{x^i=0\}$, let $\alpha_x$ be a fiber geodesic from $a$ to $b$, and $S=\llbracket J\times\alpha\rrbracket$ (the residual wall of an ordinary $BV$-jump limit, cf.\ \cite{GMS}): base rank $n-1$. Then
\[
S\big(d\omega_{i\alpha}(\varphi)\big)=\pm\int_J\varphi(x)\,(b^\alpha-a^\alpha)\,d\HH^{n-1}(x)\neq0
\]
in general --- exactly the term cancelling the distributional jump of the graph part across $J$. So piecewise-smooth jump configurations are \emph{not} silenced, and Lemma~\ref{lem:key} genuinely spends the equality-cone collapse of Theorem~\ref{thm:comass}: the threshold is base rank $\le n-2$, and calibration buys rank one.
\end{example}

\begin{theorem}[Direct $W^{1,1}$ regularity]\label{thm:w11}
Let $T\in\mathcal K_m$, $m\ge2$, with the section $U$ from Theorem~\ref{thm:mcs}. Then
\[
U\in W^{1,1}(S^n;TS^n),\qquad |U|=1\ \text{a.e.},
\]
and the weak derivative of $U$ equals its approximate derivative, which by \eqref{eq:restangent} is the calibrated block
\begin{equation}\label{eq:calblock}
\nabla_XU=\lambda\,(X-\langle X,U\rangle U)\quad\text{for all }X,\qquad \lambda\in L^{2m}(S^n).
\end{equation}
\end{theorem}
\begin{proof}
By Theorem~\ref{thm:mcs}(vi), in every base chart
\[
D_iU^\alpha
=
D_i^{\rm ap}U^\alpha\,\mathcal L^n+\mu_i^\alpha,
\qquad
\mu_i^\alpha(\varphi)
=
T_{\rm res}
\left(
\varphi\,dy^\alpha\wedge\widehat{dx^i}
\right).
\]
The horizontal-$n$ term of $d\omega_{i\alpha}(\varphi)$ vanishes on
$T_{\rm res}$.  Hence Lemma~\ref{lem:key} gives
$\mu_i^\alpha=0$ for every $i,\alpha$.  Thus
$U\in W^{1,1}_{\rm loc}$ and its weak derivative equals
$D^{\rm ap}U$.  The local statements glue because the $U^\alpha$ are
ambient components of one $\R^{n+1}$-valued map; the graph lies in $E$,
so $|U|=1$ and $x\cdot U=0$ a.e.

Calibration of the graph tangent, Theorem~\ref{thm:mcs} and \eqref{eq:restangent}, now identifies the weak covariant derivative as
\[
\nabla_XU=\lambda(X-\langle X,U\rangle U).
\]
The calibrated graph density is $(1+\lambda^2)^m$, and finite mass gives
$\lambda\in L^{2m}(S^n)$.
\end{proof}

\begin{remark}
The logical order matters: the identity ``$0=T(d\omega)$'' uses closedness of the \emph{full} current --- which holds --- and never a closedness claim for $T_\gr$ alone --- which would be false ($\partial T_\gr=-\partial T_\res$ is precisely what remains to be controlled). Lemma~\ref{lem:key} is what makes the full-current tests see only the graph part.
\end{remark}

\section{Rigidity of calibrated Sobolev sections}\label{sec:rigidity}

\begin{lemma}[Weak Bochner, contracted Ricci, and a product rule]\label{lem:weakcalc}
Let $M$ be a closed smooth Riemannian manifold of dimension $s+1$, with $s\ge4$, and use the convention $\Delta=\operatorname{div}\nabla$.  For $f\in W^{2,s}(M)$, the identities
\begin{align}
\frac12\Delta|\nabla f|^2
&=|\nabla^2f|^2+\langle\nabla f,\nabla\Delta f\rangle+\operatorname{Ric}(\nabla f,\nabla f),\label{eq:weakbochner}\\
\operatorname{div}(\nabla^2f)
&=d(\Delta f)+\operatorname{Ric}(\nabla f,\cdot)\label{eq:weakricci}
\end{align}
hold in distributions.  In \eqref{eq:weakbochner}, the middle term is understood as
\[
\langle\nabla f,\nabla\Delta f\rangle
=\operatorname{div}((\Delta f)\nabla f)-(\Delta f)^2.
\]
Moreover, if $F\in L^s(TM)$, $\operatorname{div}F\in L^{s/2}$ and $Z\in W^{1,s}(TM)\cap L^\infty$, then
\begin{equation}\label{eq:weakproduct}
\operatorname{div}(F\otimes Z)=(\operatorname{div}F)Z+\nabla_FZ
\end{equation}
in distributions, and both terms on the right lie in $L^{s/2}$ whenever $F\cdot\nabla Z\in L^{s/2}$.
\end{lemma}
\begin{proof}
Work in a finite smooth atlas and use a partition of unity.  Smooth
approximation in $W^{2,s}$ gives $f_k\to f$ in $W^{2,s}$.  Since the
ambient dimension is $s+1$, the Sobolev embedding
$W^{1,s}\hookrightarrow L^{s(s+1)}$ implies
$\Delta f_k\,\nabla f_k\to\Delta f\,\nabla f$ in some $L^q$ with $q>1$;
all remaining quadratic terms converge in $L^{s/2}$.  The smooth Bochner and
contracted Ricci identities therefore pass to distributions, with the middle
Bochner term interpreted in the displayed divergence form.

For \eqref{eq:weakproduct}, choose smooth approximants $Z_k$ converging in
$W^{1,s}$ and bounded uniformly in $L^\infty$ (local convolution followed by
truncation, then patching).  Test the smooth product rule against a smooth
field.  The term $(\operatorname{div}F)Z_k$ converges in $L^1$ because
$\operatorname{div}F\in L^{s/2}$ and
$Z_k\to Z$ in $L^{(s/2)'}$, while
$F\cdot\nabla Z_k\to F\cdot\nabla Z$ in $L^{s/2}$.  The uniform boundedness
also gives $(\operatorname{div}F)Z\in L^{s/2}$.
\end{proof}

\begin{theorem}[Sobolev rigidity]\label{thm:rigidity}
Let $m\ge2$, $n=2m+1$, and let $U\in W^{1,1}(S^n;TS^n)$, $|U|=1$ a.e., satisfy
\[
\nabla_XU=\lambda\bigl(X-\langle X,U\rangle U\bigr)
\]
weakly for every smooth $X$, with $\lambda\in L^{2m}(S^n)$.  Then there is a single pole $p\in S^n$ such that $U=R_p$ a.e.
\end{theorem}
\begin{proof}
Put $s:=2m=n-1$ and $P:=I-U\otimes U$.

\emph{Step 1: a global potential.}
The equation gives $\nabla U=\lambda P\in L^s$, hence
$U\in W^{1,s}$.  The one-form $u:=U^\flat$ is weakly closed because
$\nabla U$ is symmetric.  Moreover,
$\operatorname{div}U=s\lambda\in L^s$ and has integral zero.  Solve the
zero-mean equation
\[
\Delta f=\operatorname{div}U,
\qquad \Delta=\operatorname{div}\nabla.
\]
Elliptic regularity gives $f\in W^{2,s}$.  The one-form
$(U-\nabla f)^\flat$ is weakly closed and coclosed, hence weakly harmonic.
By elliptic regularity and $H^1_{\mathrm{dR}}(S^n)=0$, it vanishes.  Thus
$df=u$.  Since $|df|=1$ a.e., $f$ has a Lipschitz representative, and
\begin{equation}\label{eq:hessf}
\nabla f=U,
\qquad
\nabla^2f=\lambda P\in L^s,
\qquad
\Delta f=s\lambda.
\end{equation}
Thus $f\in W^{2,s}(S^n)$.

\emph{Step 2: the weak Riccati equation.}  Since $|\nabla f|^2=1$, the weak Bochner identity \eqref{eq:weakbochner}, together with $|\nabla^2f|^2=s\lambda^2$ and $\operatorname{Ric}=s g$ on the unit sphere, yields
\begin{equation}\label{eq:Ulambda}
U\lambda=-(1+\lambda^2)
\quad\hbox{in distributions}.
\end{equation}
Here $U\lambda$ means $\operatorname{div}(\lambda U)-\lambda\operatorname{div}U$; this is legitimate because $\operatorname{div}U=s\lambda$.  Consequently
\begin{equation}\label{eq:divlambdaU}
q:=\operatorname{div}(\lambda U)
=U\lambda+\lambda\operatorname{div}U
=-1+(s-1)\lambda^2\in L^{s/2}.
\end{equation}

Apply the contracted Ricci identity \eqref{eq:weakricci} to \eqref{eq:hessf}.  In one-form notation,
\begin{equation}\label{eq:contracted-sub}
\nabla\lambda-\operatorname{div}(\lambda U\otimes U)
=s\nabla\lambda+sU.
\end{equation}
The product rule \eqref{eq:weakproduct}, with $F=\lambda U$ and $Z=U$, gives
\[
\operatorname{div}(\lambda U\otimes U)
=qU+\lambda\nabla_UU=qU,
\]
because $\nabla_UU=\lambda PU=0$.  Substituting \eqref{eq:divlambdaU} in \eqref{eq:contracted-sub} gives
\[
-(s-1)\nabla\lambda
=\bigl(s-1+(s-1)\lambda^2\bigr)U.
\]
Therefore
\begin{equation}\label{eq:riccati}
\boxed{\ \nabla\lambda=-(1+\lambda^2)U\ }
\quad\text{a.e.},
\end{equation}
and in particular $\lambda\in W^{1,s/2}(S^n)$.  This derivation avoids multiplying an a priori distributional derivative of $\lambda$ by the rough projection $P$.

\emph{Step 3: the distance and the pole.}  Let $r:=\operatorname{arccot}\lambda\in(0,\pi)$.  The Sobolev chain rule and \eqref{eq:riccati} give
\[
\nabla r=-\frac{\nabla\lambda}{1+\lambda^2}=U.
\]
Define
\[
p(x):=\cos r(x)\,x-\sin r(x)\,U(x)\in\R^{n+1}.
\]
The preceding regularity gives $p\in W^{1,s/2}$.  For a tangent vector $X$, the weak product and chain rules, the Gauss formula $D_XU=\nabla_XU-\langle X,U\rangle x$, and $\lambda=\cot r$ yield
\[
D_Xp=(\cos r-\sin r\,\cot r)\bigl(X-\langle X,U\rangle U\bigr)=0.
\]
Hence $p$ is a.e.\ constant on the connected sphere.  Also $|p|=1$ and $\langle p,x\rangle=\cos r$, so $r=\operatorname{dist}(x,p)$ a.e.\ and
\[
U=\frac{\cos r\,x-p}{\sin r}=R_p
\quad\text{a.e.}
\]
\end{proof}

\begin{remark}
The rigidity proof uses only the scalar calibrated block and the integrability $\lambda\in L^{n-1}$ supplied by finite calibrated graph mass.  The restriction $m\ge2$ enters earlier, in the equality-cone classification: for $m=1$ the finite equality cone contains the conformal Hopf blocks and does not reduce to the scalar equation above.
\end{remark}

\section{Proof of the strict gap}\label{sec:proof}

\begin{proof}[Proof of Theorem~\ref{thm:main-intro}]
Fix $m\ge2$.  We first show that the calibrated floor class is empty.  Suppose
that $T\in\mathcal K_m$.  Decompose $T=T_\gr+T_\res$ as in
Theorem~\ref{thm:mcs}; the two tangent types are fixed by
\eqref{eq:restangent}.  Lemma~\ref{lem:key} and Theorem~\ref{thm:w11} show
that the graph section $U$ lies in $W^{1,1}$, has the calibrated scalar block,
and satisfies $\lambda\in L^{2m}$.  Theorem~\ref{thm:rigidity} gives
$U=R_p$ a.e.\ for one pole $p$.  The multiplicity-one graph description in
Theorem~\ref{thm:mcs} therefore identifies $T_\gr$ with the radial graph
current $T_p$.  Proposition~\ref{prop:radial} gives
\[
\partial T_\gr=-(F_p+F_{-p}).
\]
Each pole fiber is vertical, so every local contact form $b_i$ annihilates this
boundary.  Lemma~\ref{lem:identity}, applied on a finite adapted-frame cover,
then gives $T_\res=0$, and hence $T=T_p$.  This is impossible: every element
of $\Cm(\Lambda_H)$ is a cycle, whereas $\partial T_p\ne0$.  Therefore
\[
\mathcal K_m=\varnothing.
\]

By Proposition~\ref{prop:compactness}, the relaxed minimum is attained, and
the comass inequality gives $m_{\mathrm{rel}}\ge c_m\vol(S^n)$.  Equality
would put a minimizer in $\mathcal K_m$, so
\[
m_{\mathrm{rel}}>c_m\vol(S^n).
\]
Finally, $\binom{2m}{m}>2^m$, hence $c_m<2^m$.  A Hopf graph belongs to
$\mathcal G_m(\Lambda_H)$, and therefore
\[
m_{\mathrm{rel}}\le2^m\vol(S^n)<\Lambda_H.
\]
Every smooth graph either has mass at most $\Lambda_H$, in which case its mass
is at least $m_{\mathrm{rel}}$, or has mass greater than
$\Lambda_H>m_{\mathrm{rel}}$.  Thus the smooth infimum is at least
$m_{\mathrm{rel}}$ and is strictly larger than $c_m\vol(S^n)$.
\end{proof}

\section{The flat-boundary floor and radial repulsion}\label{sec:floor}

Theorem~\ref{thm:closed} established that the angular form $\Theta$ is
itself closed.  Invariantly, $\Theta=\iota_A\Omega$: the form $\Theta$ has no
$a$-factor.  Thus $\Theta$ is a globally defined $d$-form on $E$, and
$d\Theta=0$ is equivalent to the geodesic-flow invariance
$\mathcal L_A\Omega=0$.  Closed comass-one forms in the fiber direction are
the natural dual objects for pricing fillings of fiber boundaries; this
section uses $\Theta$ for that purpose.

\begin{lemma}\label{lem:theta-comass}
$\|\Theta\|_*=1$ on all $d$-planes of $TE$, and for every Lipschitz $f$ on $S^n$,
$\|\pi^*df\wedge\Theta\|_*\le \operatorname{Lip}(f)$.
\end{lemma}
\begin{proof}
Every monomial of $\Theta$ uses only the $b_i,c_i$, so $\Theta$ annihilates any argument proportional to $A$; for a unit simple $d$-vector $\xi$, replacing each factor by its orthogonal projection into the contact hyperplane $\mathscr D=\ker a$ changes nothing and does not increase the norm, so we may take $\xi=P\subset\mathscr D$ simple with $\|P\|\le1$. Then $A\wedge P$ is simple of norm $\le1$ and $\Theta(P)=\Omega(A\wedge P)$, so $|\Theta(P)|\le\|\Omega\|_*=1$ by Theorem~\ref{thm:comass}; a fiber $d$-plane attains equality ($\Theta=\tau_0$ there, with unit value). For the wedge: on any $(d{+}1)$-plane choose an orthonormal basis whose first vector spans the restriction of the $1$-form $\pi^*df$; expanding along that slot gives $|(\pi^*df\wedge\Theta)(\xi)|\le\|\pi^*df\|\,\|\Theta\|_*\le\operatorname{Lip}(f)$, since $\pi$ is a Riemannian submersion.
\end{proof}

\begin{theorem}[Flat-boundary floor]\label{thm:floor}
Let $m\ge1$, let $\Sigma$ be a current of finite mass and dimension $n$ in $E$, and let
$\beta:=\Fflat\big(\partial\Sigma-(F_p+F_{-p})\big)<\infty$, the flat norm taken in $E$ with the ambient-fixed pole-fiber convention of Section~\ref{sec:prelim}. Then
\begin{equation}\label{eq:floor}
\M(\Sigma)\ \ge\ \pi\,\vol(S^{2m})\ -\ \tfrac{\pi}{2}\,\beta.
\end{equation}
\end{theorem}
\begin{proof}
For $0<q<1$ set $f_q(x):=q^{-1}\arcsin\big(q\,\langle p,x\rangle\big)$, a smooth function with
\[
|\nabla f_q|=\sqrt{\tfrac{1-\langle p,x\rangle^2}{1-q^2\langle p,x\rangle^2}}\le1,
\qquad
\|f_q\|_\infty=f_q(p)=\tfrac{\arcsin q}{q}=:b_q\in[1,\tfrac\pi2),\quad b_q\uparrow\tfrac\pi2.
\]
Let $\alpha_q:=(\pi^*f_q)\,\Theta$; by Theorem~\ref{thm:closed}, $d\alpha_q=\pi^*df_q\wedge\Theta$, and Lemma~\ref{lem:theta-comass} gives $\|\alpha_q\|_*\le b_q$, $\|d\alpha_q\|_*\le1$.

\emph{Evaluation on the pole boundary.}
The restriction of $\Theta$ to the fiber over $p$ is its positively oriented
volume form, so $F_p(\Theta)=\vol(S^d)$.  Parallel transport preserves this
vertical volume form, while its endpoint current is $-F_{-p}$ in the
ambient-fixed convention; hence $F_{-p}(\Theta)=-\vol(S^d)$.  Since
$f_q(p)=b_q$ and $f_q(-p)=-b_q$,
\[
(F_p+F_{-p})(\alpha_q)
=b_q\vol(S^d)+(-b_q)(-\vol(S^d))
=2b_q\vol(S^d).
\]
\emph{Bookkeeping.} Fix $\varepsilon>0$ and a decomposition $\partial\Sigma-(F_p+F_{-p})=R+\partial S$ with $\M(R)+\M(S)\le\beta+\varepsilon$. Then
\[
\begin{aligned}
\Sigma(d\alpha_q)&=\partial\Sigma(\alpha_q)
=2b_q\vol(S^d)+R(\alpha_q)+S(d\alpha_q)\\
&\ge\ 2b_q\vol(S^d)-b_q\M(R)-\M(S)
\ \ge\ 2b_q\vol(S^d)-b_q(\beta+\varepsilon),
\end{aligned}
\]
using $b_q\ge1$. Since $\|d\alpha_q\|_*\le1$, $\M(\Sigma)\ge\Sigma(d\alpha_q)$. Let $\varepsilon\downarrow0$ and $q\uparrow1$ to obtain \eqref{eq:floor}.
\end{proof}

\begin{proof}[Proof of the radial-cost statement in Theorem~\ref{thm:floor-intro}]
Set $C_{\rm BCN}:=c_m\vol(S^{2m+1})$ and let
$L:=\liminf_k\VolS(V_k)$.  If $L=\infty$ there is nothing to prove.  If
$L<2C_{\rm BCN}$, pass to a subsequence whose masses converge to $L$; it is
eventually bounded by $2C_{\rm BCN}+1$.  Compactness gives a weak limit
$T=T_{\rm gr}+T_{\rm res}$.

Convergence in measure identifies the graph slice.  Indeed, after a further
subsequence $V_k(x)\to R_p(x)$ for a.e. $x$, and for every continuous
$\psi$ on $E$ and smooth base test $\varphi$ dominated convergence gives
\[
\int_{S^n}\varphi(x)\psi(x,V_k(x))\,dV
\longrightarrow
\int_{S^n}\varphi(x)\psi(x,R_p(x))\,dV.
\]
The left-hand side is the horizontal-Jacobian pairing of the graph current;
the residual has zero horizontal Jacobian.  Therefore the unique
multiplicity-one graph part of $T$ is $T_p$.

Since $T$ is a cycle and
$\partial T_p=-(F_p+F_{-p})$, one has
$\partial T_{\rm res}=F_p+F_{-p}$.  Theorem~\ref{thm:floor} with $\beta=0$
gives
$\M(T_{\rm res})\ge C_{\rm BCN}$.  Mass additivity and lower
semicontinuity yield
\[
\liminf_k\VolS(V_k)
\ge
\M(T_p)+\M(T_{\rm res})
\ge
2C_{\rm BCN}.
\]
This contradicts $L<2C_{\rm BCN}$ and proves the claim.  The current
$T_p+\Sigma_\gamma$ has the benchmark mass $2C_{\rm BCN}$, but no
mass-convergent smooth recovery sequence is asserted.
\end{proof}

\begin{remark}
\begin{enumerate}[label=\textup{(\roman*)},leftmargin=2.8em]
\item The floor \eqref{eq:floor} requires no structure of $\Sigma$ beyond
finite mass, and in particular no $\pi$-degeneracy.  It prices the boundary,
not the filling geometry: the radial graph $-T_p$ has the same boundary and
the same mass.
\item The coefficient $\pi/2$ is the one supplied by the pairing
$f\,\Theta$.  The estimate
$\|f\|_\infty\le(\pi/2)\operatorname{Lip}(f)$ is sharp on $S^n$ in the
normalization used here, and the functions $f_q$ approach equality.
\item At $m=1$, the theorem holds as stated but gives no gap on $S^3$, where
the Hopf fields already attain $c_1\vol(S^3)$ without approaching the radial
profile.
\end{enumerate}
\end{remark}

\begin{corollary}[Quantitative repulsion of the radial configuration]\label{cor:repulsion}
Let $m\ge1$ and let $T\in\Cm(\Lambda_H)$ have excess $E:=\M(T)-c_m\vol(S^{2m+1})$, with decomposition $T=T_\gr+T_\res$ as in Theorem~\ref{thm:mcs}. Then for every pole $p$,
\[
\M(T)\ \ge\ 2\,c_m\vol(S^{2m+1})\ -\ \Big(1+\tfrac\pi2\Big)\,\Fflat\big(T_\gr-T_p\big),
\]
equivalently
\[
\inf_{p\in S^n}\ \Fflat\big(T_\gr-T_p\big)\ \ge\ \frac{c_m\vol(S^{2m+1})-E}{1+\pi/2}.
\]
\end{corollary}
\begin{proof}
The flat norm is contracted by the boundary.  Indeed, if
$X=R+\partial S$ and
$\M(R)+\M(S)\le\Fflat(X)+\varepsilon$, then
$\partial X=\partial R$.  The decomposition
$\partial X=0+\partial R$ gives
\[
\Fflat(\partial X)\le\M(R)\le\Fflat(X)+\varepsilon.
\]

Put $\sigma:=\Fflat(T_\gr-T_p)$.  Since $d\Omega=0$ and
$\|\Omega\|_*=1$, pairing with any flat decomposition gives
\[
|T_\gr(\Omega)-T_p(\Omega)|\le\sigma.
\]
The radial graph is calibrated, so
$T_p(\Omega)=\M(T_p)=c_m\vol(S^{2m+1})$.  Consequently,
\[
\M(T_\gr)\ge T_\gr(\Omega)
\ge c_m\vol(S^{2m+1})-\sigma.
\]

Since $\partial T=0$ and $\partial T_p=-(F_p+F_{-p})$,
\[
\Fflat\bigl(\partial T_\res-(F_p+F_{-p})\bigr)
=
\Fflat\bigl(\partial(T_\gr-T_p)\bigr)
\le\sigma.
\]
Theorem~\ref{thm:floor} therefore gives
\[
\M(T_\res)
\ge c_m\vol(S^{2m+1})-\frac\pi2\,\sigma.
\]
Adding the two mass bounds proves the claim.
\end{proof}

\begin{remark}
At $\sigma=0$ the corollary recovers the radial-completion current value $2c_m\vol(S^{2m+1})$ with no loss; this calibrates the inequality at the current level, and smooth sharpness is not asserted; the value coincides with the localized current value $2c_m\vol(S^{2m+1})$ of Theorem~\ref{thm:floor-intro}. It is the \emph{near-radial branch} of a quantitative dichotomy: excess below $c_m\vol(S^{2m+1})$ forces the graph part flat-far from every radial graph. An explicit positive gap constant would follow from the complementary \emph{far branch} --- a quantitative statement that graph parts far from all radial graphs also force excess --- which is the open quantitative problem formulated in Problem~\ref{prob:farbranch}.
\end{remark}

\section{Consequences, limitations, and open quantitative questions}\label{sec:remarks}

\subsection{What the theorem does and does not prove}
Theorem~\ref{thm:main-intro} proves the existence, for every $m\ge2$, of a
number $\delta_m>0$ such that
\[
\VolS(V)\ge
c_m\vol(S^{2m+1})+\delta_m
\]
for every smooth unit field $V$.  The argument gives no computable lower
bound for $\delta_m$: positivity follows from compactness and the exclusion
of calibrated cycles.

The result is also not a classical Lavrentiev theorem.  The weak closure of
smooth section graphs remains cycle-constrained, and its minimum is shown to
lie strictly above the calibration charge.  The singular radial graph attains
the pointwise charge but is not a cycle.  It is therefore a comparison object,
not an admissible minimizer of a second variational problem.

\subsection{The radial benchmark}
Theorem~\ref{thm:floor-intro} shows that the radial basin has current-level
cost at least
\[
2c_m\vol(S^{2m+1}).
\]
This lower bound is sharp among finite-mass current fillings of the pole
fibers, because the parallel cylinder has exactly the required additional
mass.  Smooth sharpness remains open.  In particular, the present paper does
not identify the smooth infimum.

\begin{problem}[Explicit far-branch estimate]\label{prob:farbranch}
Find an explicit function $\Psi_m:(0,\infty)\to(0,\infty)$ such that every
weak section-graph limit $T=T_{\rm gr}+T_{\rm res}$ satisfies
\[
\inf_{p\in S^n}\Fflat(T_{\rm gr}-T_p)\ge\varepsilon
\quad\Longrightarrow\quad
\M(T)-c_m\vol(S^n)\ge\Psi_m(\varepsilon).
\]
Combined with Corollary~\ref{cor:repulsion}, such an estimate would give an
explicit gap constant.
\end{problem}

\subsection{The three-dimensional bifurcation}
The strict step in the equality-cone proof is the monotonicity of the moment
$G(R)$, which fails when $m=1$.  The equality cone then contains nonscalar
conformal blocks, including the Hopf fields.  This agrees with the known
three-dimensional minimization theorem \cite{GZ,GilM22}.

\section{Correction to the former recovery argument}\label{app:error}
Earlier versions claimed a smooth recovery sequence with mass tending to the
BCN value.  The decisive far-tube estimate treated the patched field as if
only one derivative direction became large.  In the anti-alignment region,
several transverse derivatives become large simultaneously, and the full
$2m$-dimensional graph Jacobian was not controlled.  The omitted concentration
is reflected by the pole-fiber boundary of the radial graph.  Theorem~\ref{thm:floor}
shows that any finite-mass current completing this boundary
costs at least
\[
c_m\vol(S^{2m+1})=\pi\vol(S^{2m}),
\]
so the current-level radial ledger is $2c_m\vol(S^{2m+1})$, not the BCN
value.  No estimate from the withdrawn recovery proof is used here.

\end{document}